\documentclass{article}

\usepackage{amssymb}
\usepackage{amsmath}
\usepackage{mathrsfs}
\usepackage{bbm}
\usepackage[round]{natbib}
\usepackage{hyperref}
\usepackage{graphicx}
\usepackage{verbatim}
\usepackage{mdwlist}

\makeatletter
\def\NAT@spacechar{~}
\makeatother

\newcounter{subsection1}[section]
\newtheorem{defn}[subsection1]{Definition}
\newtheorem{lemma}[subsection1]{Lemma}
\newtheorem{conv}[subsection1]{Remark}
\newtheorem{remark}[subsection1]{Remark}
\newtheorem{example}[subsection1]{Example}

\newtheorem{normaltheorem}[subsection1]{Theorem}

\numberwithin{equation}{section} 
\numberwithin{subsection1}{section}

\newenvironment{proof}{\vspace{1ex}\noindent{\textsc{Proof:}}\hspace{0.5em}}{\hfill\qed\vspace{1ex}}

\DeclareMathOperator{\llip}{lip^{\downarrow}} 

\DeclareMathOperator{\bl}{bl}

\DeclareMathOperator{\freq}{freq}

\begin{document}

\def\ra{\Rightarrow} 
\def\to{\rightarrow} 
\def\iff{\Leftrightarrow}
\def\sw{\subseteq} 
\def\mc{\mathcal} 
\def\mb{\mathbb} 
\def\sc{\,\setminus} 
\def\p{\partial} 
\def\E{\mb{E}} 
\def\P{\mb{P}}
\def\R{\mb{R}} 
\def\Q{\mb{Q}}
\def\N{\mb{N}}
\def\~{\sim}
\def\-{\,;\,} 
\def\wt{\widetilde}
\def\qed{$\blacksquare$}
\def\1{\mathbbm{1}}
\def\cadlag{c\`{a}dl\`{a}g}
\def\CPP{{C\kern-.05em\raise.23ex\hbox{+\kern-.05em+}}}

\def\Ka{($\mathscr{K}1$)}
\def\Kb{($\mathscr{K}2$)}
\def\Kc{($\mathscr{K}4$)}
\def\Kg{($\mathscr{K}3$)}
\def\Ke{($\mathscr{K}5$)}

\def\slfv{S$\Lambda$FV}
\def\l{\left}
\def\r{\right}

\allowdisplaybreaks

\author{Nic Freeman\thanks{Mathematical Institute, Oxford University. Email: \textit{nicholas.freeman@st-annes.ox.ac.uk}}}
\title{The number of non-singleton blocks in $\Lambda$-coalescents with dust}
\date{\today \textit{ (draft)}}

\maketitle

\begin{abstract}
Let $\Lambda$ be a finite measure on $[0,1)$ and set $\mu^n=\int_0^1x^{n}\Lambda(dx)$. Recall that a block $b$ in a partition of $\N$ is said to be a singleton if $|b|=1$. It is known that the blocks of the $\Lambda$-coalescent are either singletons or infinite sets which comprise a non-trivial proportion of the total population. Let $N^s_t$ and $N^a_t$ denote the number of singleton blocks and non-singleton blocks, respectively, in the $\Lambda$-coalescent at time $t>0$.

\cite{P1999} proved that $N^s_t=\infty$ if $\mu^{-1}<\infty$, and that $N^s_t=0$ if $\mu^{-1}=\infty$. \cite{S2000} gave a necessary and sufficient condition for $N^a_t+N^s_t$ to be finite. Hence, when $\mu^{-1}=\infty$, Schweinsberg's result determines if $N^a_t$ is finite or infinite. In this paper we complete the picture and show that, when $\mu^{-1}<\infty$, a third dichotomy occurs; if $\mu^{-2}=\infty$ then $N^a_t=\infty$, and if $\mu^{-2}<\infty$ then $N^a_t<\infty$.

Our proof uses the connection between $\Lambda$-coalescents and stochastic flows of bridges, which was established by \cite{BL2003}. 
\end{abstract}


\fontsize{11.3pt}{15}
\selectfont



\section{Introduction}


The $\Lambda$-coalescent is a Markov process, whose state space is the set $\mc{P}_\N$ of partitions of $\N$. It can be thought of a system of particles, which start out separated and coagulate together over time. The $\Lambda$-coalescent generalizes the well known coalescent of \cite{K1982}, and was introduced by \cite{DK1999}, \cite{P1999} and \cite{S1999} (all of which appeared independently). The name `$\Lambda$-coalescent' comes from the formulation of \cite{P1999}, which we now describe.

Let $\mc{P}_n$ denote the set of partitions of $\{1,2,\ldots,n\}$. Let $\iota_n$ denote the natural restriction map $\iota_n:\mc{P}_\N\to\mc{P}_n$, which is defined by simply removing all elements $m\in \N\sc\{1,\ldots,n\}$ from a partition of $\N$. For example, $\iota_3(\{(1,2,6),(3,5),(4),\ldots\})=\{(1,2),(3)\}$. Let $1^\N$ denote the partition of $\N$ into singletons.

If $\pi$ is a partition of $\N$, then each element of $\pi$ is known as a \textit{block}. We write $n\stackrel{\pi}{\sim}{m}$ to mean that $n$ and $m$ are in the same block of $\pi$. If $\pi=\{b_1,b_2,\ldots b_l\}$ (resp. $\pi=\{b_i\-i\in\N\}$) and $I\sw \{1,\ldots,l\}$ (resp $I\sw\N$) then the partition obtained from $\pi$ by \textit{merging} $\{b_i\-i\in I\}$ is given by $\{b_i\-i\notin I\}\cup\{\cup_{i\in I} b_i\}$ .

\begin{defn}\label{lambdacoaldef} Let $\Lambda$ be a finite measure on $[0,1]$. The $\Lambda$-coalescent is a $\mc{P}_\N$-valued Markov process $(\Pi_t)_{t\geq 0}$ such that, for all $n\in\N$, $\Pi^{(n)}_t=\iota_n(\Pi_t)$ is a $\mc{P}_n$-valued Markov chain with initial state $1^\N$ and the following dynamics: Whenever $\Pi^{(n)}_t$ is a partition consisting of $i$ blocks, the rate at which any $k$-tuple of blocks merges is
\begin{equation}
\lambda_{i,k}=\int_0^1 x^{k-2}(1-x)^{i-k}\Lambda(dx),\label{iii14}
\end{equation}
independently of all other $k$-tuples. 
\end{defn}

It is natural to regard the $\Lambda$-coalescent as a system of particles (one for each $n\in\N$), which start out separated and which merge together over time. The precise formulation of Definition \ref{lambdacoaldef} is due to \cite{P1999}. 

If $\Lambda(\{0\})=0$, the formula \eqref{iii14} is more intuitively written as $\lambda_{b,k}=\int_0^\infty x^{k}(1-x)^{i-k}x^{-2}\Lambda(dx)$. The term $x^{-2}\Lambda(dx)$ corresponds to a measure controlling the rate at which a proportion $x\in (0,1]$ of the blocks currently present merge to form a new block. When this occurs we call it a \emph{coagulation event}. The remaining `binomial' term $x^k(1-x)^{i-k}$ says that, of the first $i$ blocks, each block chooses independently whether to become part of the new block, or remain alone (with probabilities $x$ and $1-x$ respectively). 

\begin{remark}
The terminology `coagulation event' comes from thinking of the $\Lambda$-coalescent as a system of particles. The particles, one for each $n\in\N$, start out separated and and coagulate together over time; a merger of blocks corresponds to a group of particles coagulating, forming a single block/particle.
\end{remark}

If $\Lambda(\{1\})>0$, then corresponding events occur at rate $\Lambda(\{1\})$, which coagulate the whole population into a single block. From a theoretical point of view, this adds no extra complexity and serves only to obfuscate the behaviour which we are trying to capture. Results concerning $\Lambda$-coalescents for which $\Lambda(\{1\})=0$ are easily extended to general $\Lambda$, by simply superimposing the extra coagulation events. With this in mind: 

\begin{conv}
In this article, without further comment, we consider only $\Lambda$ for which $\Lambda(\{1\})=0$.
\end{conv}

If $f:\N\to\N$ is a bijection, and $\pi\in\mc{P}_n$ (where $(b_i)$ is possibly a finite sequence), then we define $f(\pi)=\{f(b)\-b\in \pi\}$, where $f(b)=\{f(k)\-k\in b\}$. Recall that a (random) partition $\pi\in\mc{P}_\N$ is said to be \textit{exchangeable} if, for any permutation $\sigma$ of $\N$, $\sigma(\pi)$ has the same distribution as $\pi$. \cite{P1999} showed that, for any $t>0$, $\Pi_t$ is an exchangeable partition.

The following lemma collects together some results, due to Kingman, which can be found in Section 2.2 of \cite{B2006}. For a finite set $A$, $|A|$ denotes the cardinality of $A$. If $A$ is infinite then we set $|A|=\infty$.

\begin{lemma}\label{exchparts} Let $\pi$ be an exchangeable partition of $\N$. Then, almost surely, for every block $b$ of $\pi$, the limit
$$\freq(b)=\lim\limits_{k\to\infty}\frac{|b\cap \{1,\ldots, k\}|}{k}$$
exists. The quantity $\freq(b)$ is known as the \emph{asymptotic frequency} of $b$. It holds that $\sum_{b\in\pi} \freq(b)\leq 1$, and also that
$$\lim\limits_{k\to\infty}\frac{|\{n=1,\ldots,k\-\{n\}\in\pi\}|}{k}\stackrel{a.s.}{=}1-\sum_{b\in\pi} \freq(b)$$
Further, almost surely, if $\freq(b)=0$ then $b$ is a singleton. Almost surely, $\pi$ has no singletons if and only if $\sum_{b\in\pi}\freq(b)=1$.
\end{lemma}

\begin{remark}
For a general subset $b\sw\N$, there is no reason for the limit $\freq(b)$ to exist. 
\end{remark}

Let us introduce some terminology, in the notation of the above lemma. If $\freq(b)>0$, then the elements of $b$ comprise a non-zero proportion $\freq(b)$ of $\N$, and we refer to $b$ as an \emph{atomic block} of $\pi$. This terminology reflects the idea that an atomic block corresponds to an atom of the probability measure which records the proportion of the population within each block.

Each singleton $\{n\}$ comprises only a null proportion of $\N$, in that $\freq(\{n\})=0$. The set of singletons of $\pi$ is called the \emph{dust} of $\pi$, and the elements of this set comprise a proportion of $\N$ which is (possibly $0$) and given by
$$\mathscr{D}(\pi)=1-\sum_{b\in\pi}\freq(b).$$

\cite{P1999} proved the following result, which establishes a dichotomy in the behaviour of the dust. Let
\begin{equation}\label{mun}
\mu^{n}=\int_0^1x^n\Lambda(dx).
\end{equation}

\begin{normaltheorem}[\citealt{P1999}]\label{pitmansthm}
If $\mu^{-1}<\infty$, then $\P\l[\forall t>0, \mathscr{D}(\Pi_t)>0\r]=1$, whereas if $\mu^{-1}=\infty$ then $\P\l[\forall t>0, \mathscr{D}(\Pi_t)=0\r]=1$. 
\end{normaltheorem}

Let
$$N^s_t=|\{b\in\Pi_t\-\freq(b)=0\}|,$$
which, in words, is the number of singletons of $\Pi_t$. By Lemma \ref{exchparts} (and the fact that $N^s_u\leq N^s_v$ when $u\geq v$), Theorem \ref{pitmansthm} implies (almost surely) that $N^s_t=\infty$ or $N^s_t=0$, and that $\mathscr{D}(\Pi_t)>0$ if and only if $N^s_t=\infty$. 

A second dichotomy, this time in the behaviour of the total number of blocks of $\Pi_t$, was proven by \cite{S2000}. Let
$$\mu^*=\sum_{i=2}^\infty\l(\sum_{k=2}^i(k-1)\binom{i}{k}\lambda_{i,k}\r)^{-1}.$$

\begin{normaltheorem}[\citealt{S2000}]\label{sthm}
If $\mu^*<\infty$ then $\P\l[|\Pi_t|<\infty\r]=1$, whereas if $\mu^*=\infty$ then $\P\l[|\Pi_t|=\infty\r]=1$.
\end{normaltheorem}

It can be shown that $\mu^{-1}<\infty$ implies that $\mu^*=\infty$; in other words, $\Lambda$-coalescents with a non-empty dust component have infinitely many blocks for all time (of course, we already knew this since non-empty dust implies $N^s_t=\infty$). 
Recall that the $\Lambda$-coalescent is said to \emph{come down from infinity} if $\P\l[\forall t>0, |\Pi_t|<\infty\r]=1$. Theorem \ref{sthm} shows that the $\Lambda$-coalescent comes down from infinity if and only if $\mu^*<\infty$.

The main result of this article is to establish a third (and, in some sense, final) dichotomy which occurs for $\Lambda$-coalescents. Let
\begin{equation*}\label{Nat}
N^a_t=|\{b\in\Pi_t\-\freq(b)>0\}|,
\end{equation*}
which is the number of atomic blocks of $\Pi_t$. Note that $$N^a_t+N^s_t=|\Pi_t|.$$ 

If $\mu^{-1}=\infty$ then (see above), $\Pi_t$ has no singletons, and Theorem \ref{sthm} tells us that $N^a_t=\infty$ when $\mu^*=\infty$ and $N^a_t<\infty$ when $\mu^*<\infty$. The behaviour of $N^a_t$, in the case $\mu^{-1}<\infty$, is described by the following theorem.

\begin{normaltheorem}\label{uscthm}
Suppose that $\mu^{-1}<\infty$. If $\mu^{-2}=\infty$ then $\P\l[\forall t>0, N^a_t=\infty\r]=1$, whereas if $\mu^{-2}<\infty$ then $\P\l[\forall t>0, N^a_t<\infty\r]=1.$
\end{normaltheorem}

With Theorems \ref{pitmansthm}-\ref{uscthm} in hand, the qualitative behaviour of $N^a$ and $N^s$ can be completely classified, as in Figure \ref{lcoalfig}. 

It is surprising that Theorem \ref{uscthm} was not discovered until now; in fact we were unable to find even a single example in the literature of a $\Lambda$-coalescent which was shown to have both $\mathscr{D}(\Pi_t)>0$ and $N^a_t=\infty$. The question `are there $\Lambda$-coalescents with $\mathscr{D}_t>0$ and $N^a_t=\infty$' naturally emerged out of the analysis in \cite{F2011a}, see Remark \ref{frem}. We will discuss our proof of Theorem \ref{uscthm} in Section \ref{proofoutline}. See Example \ref{betaex} for an application of Theorem \ref{uscthm} to $\beta$-coalescents.

\begin{remark}\label{frem}
In \cite{F2011a}, a natural extension of the $\Lambda$-coalescent is defined in a spatial continuum, and its behaviour is classified in a similar manner to that of Theorems \ref{pitmansthm}-\ref{uscthm}. It is shown there that the introduction of space enriches the behaviour of the $\Lambda$-coalescent. For example, if $p_t$ is the probability that the coalescent of \cite{F2011a} has finitely many blocks at time $t$, it is sometimes the case that, for some deterministic $t_0\in(0,\infty)$, $p_t>0$ for $t<t_0$ and $p_t=0$ for $t\geq t_0$.
\end{remark}

\begin{figure}[t]
\begin{center} 
\includegraphics[height=3in,width=5in]{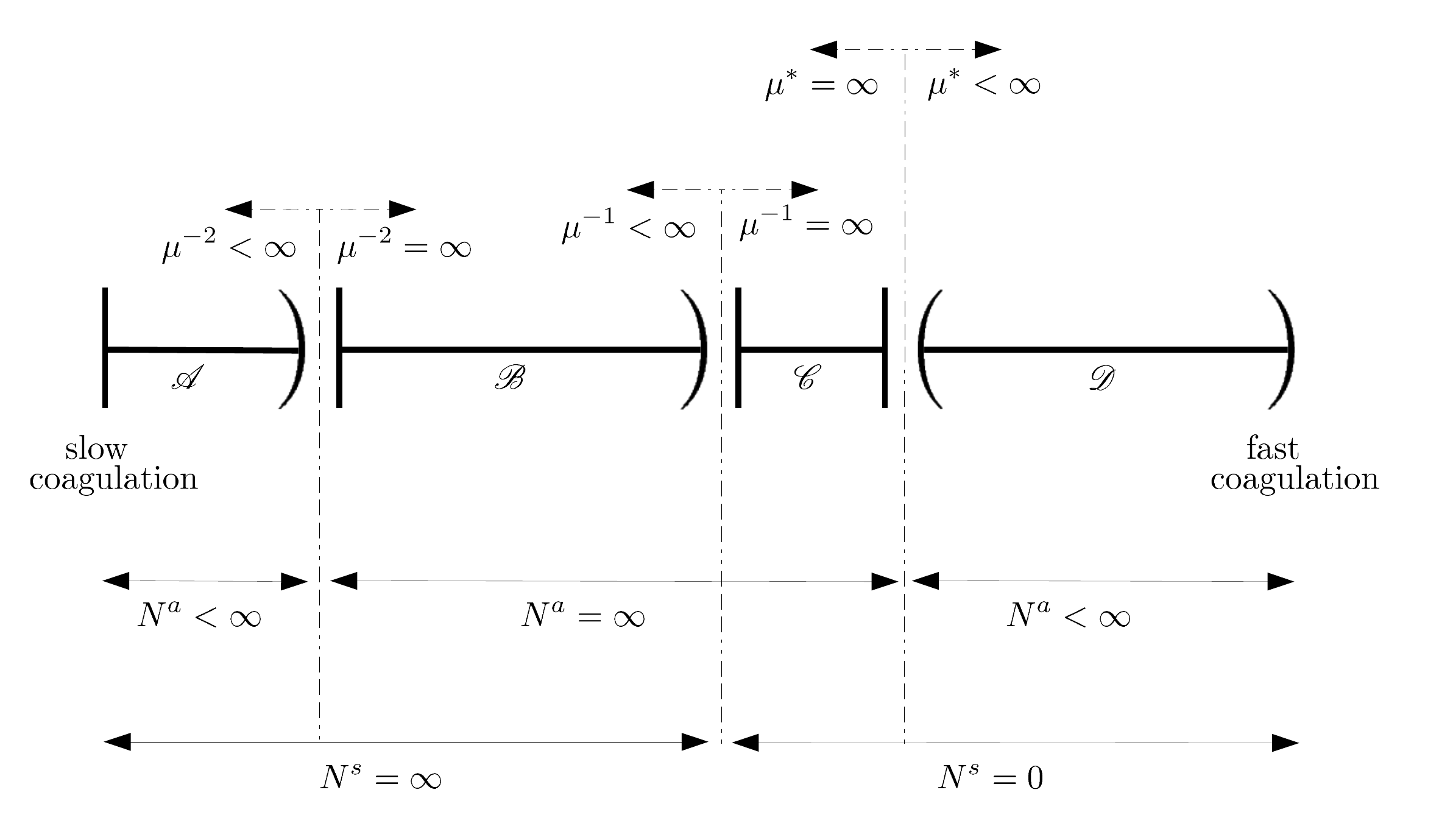} 
\caption{\label{lcoalfig} \small The qualitative changes in the behaviour of the $\Lambda$-coalescent are shown. The behaviour of the $\Lambda$-coalescent is determined solely by the behaviour of the measure $\Lambda$ at $0+$, that is, by the rate of small coagulation events. Moving from left to right, across the picture, the rate of these coagulation events increases. The resulting four behaviours are labelled $\mathscr{A}$ to $\mathscr{D}$.}
\end{center} 
\end{figure} 

\begin{example}[$\beta$-coalescents]\label{betaex} The $\Lambda$-coalescent where $\Lambda(dx)$ has the $\beta(2-\alpha,\alpha)$ distribution, for $\alpha\in(0,2)$, is known as the $\beta$-coalescent with parameter $\alpha$. The $\beta$-coalescents are one of the most tractable families of $\Lambda$-coalescents, see for example \cite{BBS2008}. 

The first example of a (non-Kingman) $\Lambda$-coalescent, which was introduced in \cite{BS1998}, corresponds to case of a $\beta$-coalescent with $\alpha=1$. In this case $\Lambda$ is the uniform probability measure on $[0,1]$. The case $\alpha=2$ is, by convention (or as the natural limiting case) Kingman's coalescent, where $\Lambda$ is a point mass at $0$. 

Label the different behaviours of the $\Lambda$-coalescent as $\mathscr{A}$ to $\mathscr{D}$, as in Figure \ref{lcoalfig}. Some easy estimates show that
\begin{itemize}
\item For $\alpha\in(0,1)$, behaviour $\mathscr{B}$ occurs
\item For $\alpha=1$, behaviour $\mathscr{C}$ occurs.
\item For $\alpha\in(1,2]$ behaviour $\mathscr{D}$ occurs.
\end{itemize}
Behaviour $\mathscr{A}$ does not occur amongst $\beta$-coalescents. However, it is easy to construct $\Lambda$-coalescents for which behaviour $\mathscr{A}$ does occur;  for example $\Lambda(dx)=x^2 dx$.
\end{example}

\subsection{Outline of the proof of Theorem \ref{uscthm}}
\label{proofoutline}

The case $\mu^{-2}<\infty$ is well known, and corresponds to the case where coagulation events (of $\Pi_t$) occur only at finite rate. Each coagulation event produces at most one new atomic block, and hence the total number of atomic blocks is always finite. We give an argument based on this intuition in Section \ref{finitesec}.

Let us now explain why the other case, $\mu^{-2}=\infty$, is not so obvious. In the case $\mu^{-2}=\infty$, it is immediate that (in some limiting sense) infinitely many coagulation events occur during all non-trivial time intervals. Since $\mu^{-1}<\infty$ (as an assumption of the theorem), at all times $\Pi_t$ contains a dust component with $\mathscr{D}(\Pi_t)>0$. Hence, each coagulation event takes a proportion of singletons out of the dust component, and merges them into a single new block. The catch is that each coagulation event also takes a proportion of the atomic blocks, and merges them too into the new block. Therefore, the argument from the case $\mu^{-2}<\infty$ does not work in reverse; it is not obvious that having infinitely many coagulation events is enough to guarantee infinitely many atomic blocks.

Note that, although the singletons decrease in number at each coagulation event, the atomic blocks can both increase and decrease in number. This is easily seen by considering $\Pi^n_t$. Therefore, it could potentially be the case that $N^a_t<\infty$ and $N^a_{t'}=\infty$ occurred at two different (random or deterministic) times $t,t'\in(0,\infty)$. With this observation in mind, it is not obvious that a result as clear cut as Theorem \ref{uscthm} holds. 

But, it is natural to suspect that for at least some choices of $\Lambda$, for at least some $t>0$, we might have $\mathscr{D}_t>0$ and $N^a_t=\infty$. Let us think, therefore, about how we might go about identifying such cases. 

One option is to consider the restrictions $\Pi^n_t$, and hope to show that as $n\to\infty$, the number of non-singleton blocks of $\Pi^n_t$ tends to infinity as $n\to\infty$. There is some difficulty involved in getting direct estimates on the behaviour of $\Pi^n_t$, especially in the general case where the behaviour of $\Lambda$ might vary wildly near $0$. We choose not to attempt these calculations, noting that there is already a wealth of literature offering more sophisticated ways of analysing the $\Lambda$-coalescent.

\begin{remark}
For example, \cite{P1999}, \cite{BL2003}, and \cite{BBC2005} offer three genuinely different ways in which to represent the $\Lambda$-coalescent. A further representation (for an important special case) is given in \cite{BBS2008}. We refer the reader to \cite{B2009} for further references. 
\end{remark}

In general, the $\Lambda$-coalescent is an infinite rate object and, in the spirit of the above paragraph, it is sensible to use more tractable finite rate approximations. As we saw in Definition \ref{lambdacoaldef}, one way to achieve this is by looking at only a finite subset of the particles (i.e. blocks) which are active in the system. A second option is to reduce the total rate of the system to something finite, and consider approximation by a sequence of coalescents, each acting on $\N$, but with only finite rate coagulation events.

To be precise, we could consider a sequence of $\Lambda$-coalescents, with
$$\int_0^1x^{-2}\Lambda^{(n)}(dx)<\infty,$$
such that $\Lambda^{(n)}\to\Lambda$, in some sense. This is borne out in the work of \cite{BL2003}, and the corresponding construction of the $\Lambda$-coalescent will be recapped in Section \ref{bridgessec}. Using this construction, we have excellent global control over the number of atomic blocks, and we will be able to get natural estimates on the number and corresponding asymptotic frequencies of the atomic blocks which appear in the finite rate approximations.

However, there is still one major problem to solve; the construction from \cite{BL2003} produces the $\Lambda$-coalescent as a limiting object, in the sense of finite dimensional distributions. Consequently, and in contrast to Definition \ref{lambdacoaldef}, the theory of \cite{BL2003} allows us to make estimates which are uniform in space (i.e. valid for arbitrarily many particles), but at the expense of uniformity in time. 

A minor technicality arises in the limit taking, since a partition with no atomic blocks, is close (in a suitable metric space sense) to a partition with a large (possibly infinite) number of very small atomic blocks. Therefore, some degree of care is needed to make sure the blocks which appear in our finite rate approximations really do appear in the limit. As one might expect, the issue is overcome by (loosely speaking) taking an arbitrary $\epsilon>0$ and working with blocks $b\in\Pi_t$ that have $\freq(b)\geq\epsilon>0$. 

Since the limit is taken in terms of finite dimensional distributions, we arrive at only the result $N^a_t=\infty$ for fixed, deterministic, times $t\in(0,\infty)$. It could, potentially, still be the case that $N^a_T<\infty$ at some random time $T$ or, even worse, on a random null (but potentially dense) subset of $(0,\infty)$.

We regain global control in time through the following observation. Fix deterministic $t'>0$ and consider the set $S$ of atomic blocks of $\Pi_{t'}$. Regard each of these blocks as a singletons in the initial state of a new coalescent, and define the evolution of our new coalescent via the coagulation induced from $(\Pi_{t'})_{t>t'}$. Let us name this new coalescent, which acts on $S$, as $\wt{\Lambda}_S$. It is not hard to see that $\wt{\Lambda}_S$ is a $\Lambda$-coalescent, with the same coagulation rates as $\Pi$. Since $\mu^{-1}<\infty$, $\Pi$ does not come down from infinity, and hence neither does $\wt{\Lambda}_S$. Hence, $N^a_{t}=\infty$ for all $t\geq t'$. Taking $t'$ to be arbitrarily small completes the proof of Theorem \ref{uscthm}.

\section{Flows of bridges}
\label{bridgessec}

In this section we recap some well known results of \cite{BL2003}, namely the construction of $\Lambda$-coalescents using flows of bridges. From now on we switch to using the measure
$$\nu(dx)=\frac{1}{x^2}\Lambda(dx),$$
since it will be intuitively clearer to see the $\Lambda$-coalescent in terms of $\nu$, rather than $\Lambda$.

\begin{remark} In terms of $\nu$, the conditions for Theorem \ref{uscthm} to apply are that $\nu(\{0,1\})=0$, and $\int_0^1 x\nu(dx)<\infty$. Note that the apparently extra assumption that $\nu(\{0\})=0$ is implied by $\int_0^1x^{-1}\Lambda(dx)<\infty$, via the convention that $1/0=\infty$. The dichotomy stated in Theorem \ref{uscthm} then rests on whether $\nu([0,1])$ is finite or infinite. 

The condition for the existence of a $\Lambda$-coalescent corresponding to $\nu$ is that $\int_0^1x^2\nu(dx)<\infty$. \end{remark}

In keeping with our new terminology, we refer to the $\Lambda$-coalescent $(\Pi_t)$ corresponding to $\nu$ as the $\nu$-coalescent, but we continue to use the term $\Lambda$-coalescent as a name for the general family of processes.

For the remainder of Section \ref{bridgessec}, let $\nu$ be a measure on $[0,1]$ such that $\int_0^1 x^2\nu(dx)<\infty$ and $\nu(\{0,1\})=0$. Let $(\Pi_t)_{t\geq 0}$ denote the $\nu$-coalescent. 

\begin{remark} In order to state the results of \cite{BL2003}, we will first need to develop some theory relating to exchangeable partitions. For this reason, it is convenient for us to simply use Sections 2.1, 2.3 and 4.4 of \cite{B2006} as our main reference for this section.
\end{remark}

\begin{defn}
Let $\mc{P}_{\mc{M}}$ be the space of sequences $s=(s_i)_{i=1}^\infty\sw [0,1]$ which are such that 
$$s_1\geq s_2\geq \ldots\geq 0\;\text{ and }\;\sum\limits_{i=1}^\infty s_i\leq 1$$
The set $\mc{P}_{\mc{M}}$ is a compact metric space, endowed with the metric $d(s,s')=\max\{|s_i-s'_i|\-i\in\N\}.$
\end{defn}

For an exchangeable partition $\pi=\{b_1,b_2\ldots\}$ of $\N$, we write $|\pi|=(\freq(b_1),\freq(b_2)\ldots)$. If $\pi=\{b_1,\ldots,b_n\}$ has only finitely many blocks, we set $|\pi|=(\freq(b_1),\ldots, \freq(b_n),0,0,\ldots)$. We write $|\pi|^{\downarrow}$ for the (random) sequence obtained by rearranged the elements of $|\pi|$ into decreasing order. The sequence $|\pi|^\downarrow$ is known as the sequence of \emph{asymptotic frequencies} of $\pi$. For any exchangeable partition $\pi\in\mc{P}_\N$, $|\pi|^\downarrow$ is an element of $\mc{P}_{\mc{M}}$. 

The following lemma will be of use to us, later on.

\begin{lemma}[Proposition 2.1, \citealt{B2006}]\label{skoruse}
Let $(s^n)$ be a sequence in $\mc{P}_{\mc{M}}$, where $s^n=(s^n_i)_{i=1}^\infty$ and $s\in\mc{P}_{\mc{M}}$. Then $d(s^n,s)\to 0$ if and only if for all $i\in\N$, $s^n_i\to s_i$.
\end{lemma}

Let $L^2[0,1]$ denote the space of (equivalence classes of) square integrable functions $f:[0,1]\to[0,1]$, equipped with the usual $L^2$ metric, $$d_{L^2}(f,g)=\l(\int_0^1 |f(x)-g(x)|^2 dx\r)^{1/2}.$$

\begin{defn}
Let $s=(s_i)_{i=1}^\infty\in \mc{P}_{\mc{M}}$, and let $s_0=1-\sum_1^\infty s_i$. Let $(U_i)$ be an infinite sequence of independent uniform random variables on $[0,1]$. Any process which has the same distribution (in $L^2[0,1]$) as 
$$b_{s}(y)=s_0y+\sum_{i=1}^\infty s_i\1\{U_i\leq y\}$$
is known as an $s$-\emph{bridge} (or, when $s$ is not specified, a bridge). By definition, when we use more than one bridge at once, we use a different independent sequence of uniform random variables for each bridge.

If $s$ is random, then by definition we choose $(U_i)_1^\infty$ to be independent of $s$, and we also say $b_s$ is an $s$-bridge.
\end{defn}

\begin{defn}
If $s=(x,0,0,\ldots)$ then we say the bridge $b_s$ is a \emph{simple bridge}, and write
$$b_s(y)=b_x(y)=(1-x)y+x\1\{U\leq y\}.$$
Here $U$ is a uniform random variable on $[0,1]$, and $U$ is independent of $s$.
\end{defn}

\begin{defn}
A bridge $b_s$ where $s=(s_1,\ldots,s_n,0,0,\ldots)$ is said to be a \emph{finite bridge}. In the case where $s$ is random, $n$ may be random and dependent on $s$ (but not on the $(U_i)$).
\end{defn}

Note that a bridge is a right continuous, strictly increasing function. It is easily checked that the composition of finitely many bridges is a bridge, and that the composition of finitely many finite bridges is a finite bridge. But note that the composition of finitely many simply bridges is only a finite bridge. It will be convenient to picture our bridges as stochastic flows, rather than functions, as indicated in Figure \ref{bridge1}.

\begin{figure}[t]
\begin{center} 
\includegraphics[height=1.5in,width=2.5in]{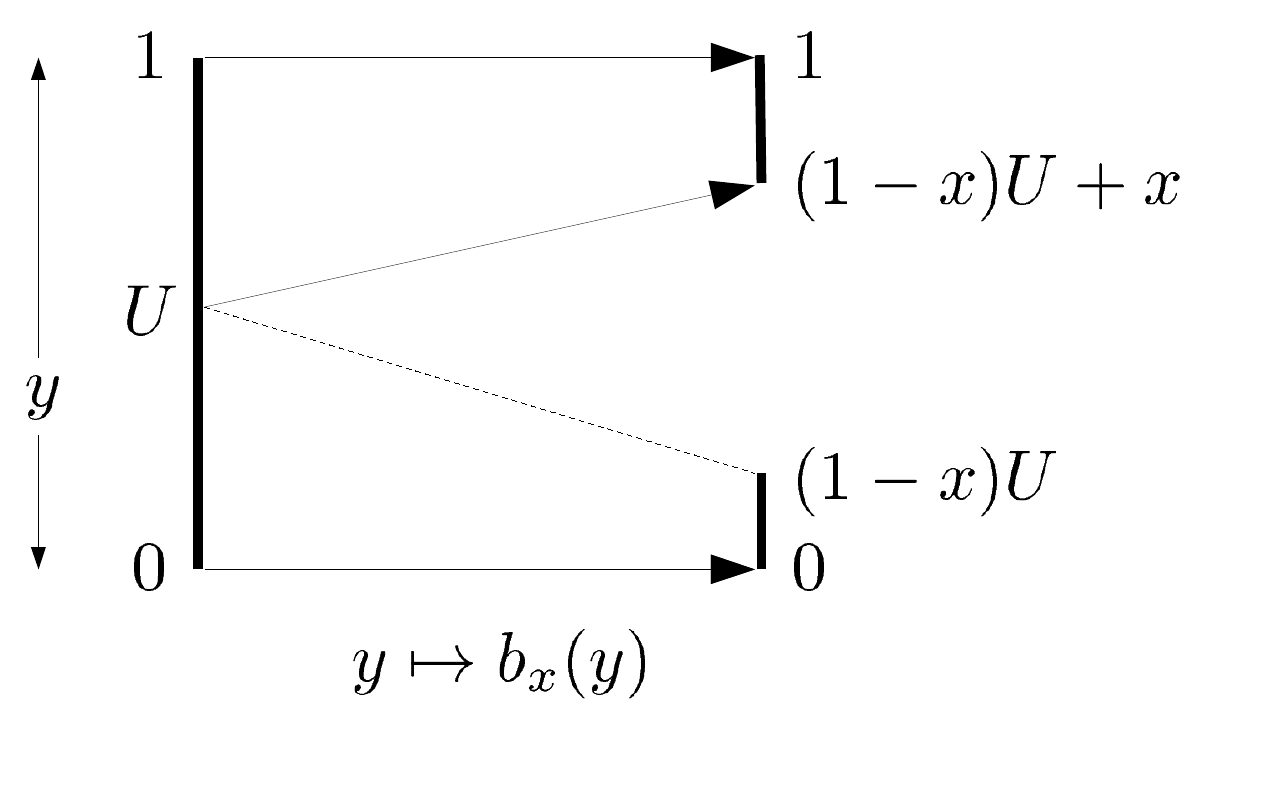} 
\caption{\label{bridge1} \small The simple bridge $b_x(y)=(1-x)y+x\1\{U\leq y\}$.}
\end{center} 
\end{figure} 

For an increasing function $f:[0,1]\to[0,1]$, let $f^{-1}$ denote the right continuous inverse of $f$, that is
$$f^{-1}(y)=\inf\{z\in[0,1]\-f(z)>y\}.$$
The following two Lemmas summarise the connection between exchangeable partitions and bridges, which is developed in Section 4.4 of \cite{B2006}.
 
\begin{lemma}\label{bridgepart}
Let $b_s$ be a bridge, and let $(V_i)_1^\infty$ be a sequence of uniform random variables on $[0,1]$, independent of each other and of $s$. Define a partition $\pi$ of $\N$ by
$$i\stackrel{\pi}{\sim}j\Longleftrightarrow b_s^{-1}(V_i)=b_s^{-1}(V_j).$$
Then $\pi$ is an exchangeable random partition. Further, $\mathscr{D}(\pi)=s_0=1-\sum_1^\infty s_i$, and $|\pi|^{\downarrow}=(s_i)_1^\infty$.
\end{lemma}

\begin{lemma}
Every exchangeable partition $\pi$ of $\N$ is equal in law to the exchangeable partition constructed by Lemma \ref{bridgepart} from $s=|\pi|^{\downarrow}$. 
\end{lemma}

We extend the definition of $\mathscr{D}$ in the natural way, as follows. If $b$ is a bridge, $\mathscr{D}(b_s)$ denotes the size of the dust in the corresponding exchangeable partition. 

\cite{BL2003} showed that, with a careful choice of bridges, it is possible to use Lemma \ref{bridgepart} to construct the $\Lambda$-coalescent (as a limit, in terms of finite dimensional distributions). To this end, let $M$ be a Poisson point process with points $(t,x)\in(0,\infty)\times [0,1]$ and intensity measure
$$dt\otimes \nu(dx).$$
Here $dt$ denotes Lebesgue measure. We now make several `without loss of generality' assumptions, since it will be convenient to have several almost sure properties of $M$ as `sure' properties of $M$.

Since $\nu(\{0,1\})=0$, almost surely $M$ is contained in $(0,\infty)\times (0,1)$, so without loss of generality we will assume that $M\sw (0,\infty)\times (0,1)$. It follows from $\sigma$-finiteness of $dt\otimes \nu(dx)$ that, almost surely, for all $t\in (0,\infty)$, at most one point of $M$ has $t$ as its first coordinate. So, without loss of generality, we will assume this is the case for all realizations of $M$. Similarly, almost surely, for each $x\in (0,1)$, at most one point of $M$ has $x$ as its second coordinate; without loss of generality we assume this, too, is the case for all realizations of $M$. For all $n\in\N$, since $\nu((1/n,1])$ is a finite measure, $M\cap (0,\infty)\times (1/n,1)$ is almost surely a finite set. So, finally, without loss of generality we assume that both $M\cap (0,\infty)\times (1/n,1)$ is finite, for all realizations of $M$ and all $n,t,t'$. 

If $\nu([0,1])=\infty$, it follows that for all $0\leq t<t'$, $M\cap(t,t']\times(0,1)$ is almost surely a countably infinite set. If this is the case, without loss of generality, we assume also that $M\cap(t,t']\times(0,1)$ is countably infinite for all $0\leq t<t'$. If $\nu([0,1])<\infty$ then for all $0\leq t<t'$, $M\cap(t,t']\times(0,1)$ is almost surely a finite set, and, without loss of generality, if $\nu([0,1])<\infty$ we assume this is so. We adopt the convention that a Poisson random variable $X\sim\text{Poisson}(\infty)$ simply means $X=\infty$.

Fix $0\leq t<t'<\infty$. Set $N_\infty=|M|$, and note that by the above, $$N_\infty\sim\text{Poisson}\big((t'-t)\nu([0,1])\big).$$ Enumerate $M\cap(t,t']\times (0,1)^2$ such that
\begin{align*}
&M\cap(t,t']\times (0,1)=\{(t_i,x_i)\-i=1,\ldots, N_\infty\}
\end{align*}
where, for all $i$, $x_i<x_{i+1}$. For each $m\in\N$ and $0\leq t<t'$, define
\begin{equation*}\label{Nm}
N_m=\#\big(M\cap(0,1)\times (1/m,1)\big)
\end{equation*}
and note that then $\{(t_i,x_i)\-i=1,\ldots,N_m\}=M\cap(0,1)\times (1/m,1)$, so as 
\begin{equation}\label{Nmdist}
N_m\sim\text{Poisson}\big((t'-t)\nu((1/n,1])\big).
\end{equation}
It is easy to see that $N_m\leq N_{m+1}$, and that $N_m\to N_\infty$ as $m\to\infty$.

For each $n\in\N\cap\{1,\ldots, N_\infty\}$, define the bijection $\sigma:\{1,\ldots, n\}\to\{1,\ldots, n\}$ by the properties
\begin{align*}
&\{(t_i,x_i)\-i=1,\ldots,n\}=\{(t_{\sigma(i)},x_{\sigma(i)})\-i\in\N\}\\
&\text{for all }i,\;t_{\sigma}(i)< t_{\sigma(i+1)}
\end{align*}
To avoid excessive subscripts, we do not normally write the dependence of $\sigma=\sigma_n$ on $n$. Set
\begin{equation}\label{Bnt}
B^{(n)}_{t,t'}=b_{x_{\sigma(1)}}\circ\ldots \circ b_{x_{\sigma(n)}}.
\end{equation}
If $N_\infty<\infty$ and $n\in\N$, $n>N_\infty$, then we define $B^{(n)}_{t,t'}=B^{(N_\infty)}_{t,t'}$. Note that we have defined $B^{(n)}_{t,t'}$ for all $n\in\N$, but not for $n=\infty$. The limiting case $n=\infty$ will be defined as part of Theorem \ref{lconstrbl}. For all $n\in\N$, $B^{(n)}_{t,t'}$ is a finite bridge.

\begin{remark}
Our notation for the order of function composition is that $(f \circ g)(x)=g(f(x))$. This is the same as in \cite{B2006}.
\end{remark}

\begin{remark}
To summarise the above, we impose an order on $M=\{(t_i,x_i)\-i\in\N\}$ by ranking the second coordinates of elements of $M$ in decreasing order. Then, $\{(t_{\sigma(i)},x_{\sigma(i)})\-i\in\N\}$ is the first $n$ elements of $M$, reordered by time coordinate. 

The random function $B^{(n)}_{t,t'}$ is the composition of the simple bridges corresponding to the first $n$ points of $M$, composed according to the order of their time coordinates. Similarly, the random function $B^{(N_m)}_{t,t'}$ is the composition, after appropriate reordering, of the bridges corresponding to $M\cap (t,t')\times (1/n,1)$.

\end{remark}

We denote the right continuous inverse of $B^{(n)}_{t,t'}$ by $B^{(n),-1}_{t,t'}$. 

\begin{normaltheorem}\label{lconstrbl} The following holds.
\begin{enumerate}
\item For each $0\leq t<t'$, the sequence $B^{(N_m)}_{t',t}$ converges in law. 
\suspend{enumerate}
Denote the limit by $B^{(\infty)}_{t,t'}$, and its right continuous inverse by $B^{(\infty),-1}_{t,t'}$.

Now fix $t>0$. Let $(V_i)_{i=1}^\infty$ be a sequence of uniform random variables on $[0,1]$, which are independent of $\{B_{0,t}^{(n)}\-n\in\N\cup\{\infty\}\}$, and of each other. For each $n\in\N\cup\{\infty\}$ and $t>0$ define a random partition $\pi_t^{(n)}$ of $\N$ by
$$i\stackrel{\pi^{(n)}_t}{\sim} j\Longleftrightarrow B^{(n),-1}_{0,t}(V_i)=B^{-1}_{0,t}(V_j).$$
\resume{enumerate}
\item Then, $\pi^{(\infty)}_t$ and $\Pi^{(\infty)}_t$ have the same distribution.
\item As $m\to\infty$, $\pi^{(N_m)}_t$ converges to $\pi^{(\infty)}_t$ in distribution.
\item As $m\to\infty$, $|\pi^{(N_m)}_t|^\downarrow$ converges to $|\pi^{(\infty)}_t|^\downarrow$ in distribution.
\end{enumerate}
\end{normaltheorem}
\begin{proof} Essentially, this result is Theorem 4.3 of \cite{B2006}. Note, however, that \citeauthor{B2006} defines the symbol $B^{(n)}$ slightly differently to us (see \eqref{Bnt} for our definition). Rather than considering each point of $M$ as a separate entity, Bertoin's $B^{(n)}_{t,t'}$ is a composition of the bridges corresponding to all points of the finite set $M\cap (t,t']\cap (1/n,1)$. 

The connection is that, for $m\in\N$,  \citeauthor{B2006}'s $B^{(m)}_{t,t'}$ is our $B^{(N_m)}_{t,t'}$. Having realized this, statements \textit{1} and \textit{2} are simply restatements of facts contained within Theorem 4.3 of \cite{B2006}. Statement \textit{3} is proved during the course of the proof of that theorem, and statement \textit{4} follows immediately from \textit{3}, by applying Proposition 2.9 of \cite{B2006}.
\end{proof}


\begin{remark}


Theorem 4.3 of \cite{B2006} also states that $B_{t,t'}$ is a stochastic flow (in some sense), but Theorem \ref{lconstrbl} contains all the information needed for our purposes. \end{remark}


\section{Proof of Theorem \ref{uscthm}}
\label{proofsec}

For the remainder of the article, suppose that $\int_0^1 x\nu(dx)<\infty$ and that $\nu(\{0,1\})=0$. We consider the case $\nu([0,1])<\infty$ in Section \ref{finitesec}, and the remaining sections are devoted to the case $\nu([0,1])=\infty$.

A central concept of the proof is the idea of a hole in a finite bridge, defined as follows.

\begin{defn}\label{holesdef}
A set $H\sw[0,1]$ is said to be a \emph{hole} of the finite bridge $b_s$ if $H$ is a maximally connected component of $[0,1]\sc b_s([0,1])$.
\end{defn}

Since $b_s$ is right continuous and has only finitely many discontinuities, $b_s$ has finitely many holes, each of which is a half open interval $H=[h_1,h_2)$. The \emph{size} of the hole $H=[h_1,h_2)$ is given by 
$$\mathscr{S}(H)=h_2-h_1.$$
See Figure \ref{bridge3}. Note that $\mathscr{S}(H)>0$ for all holes $H$ of $b_s$. We say $h_2$ is the \emph{boundary point} of the hole $H=[h_1,h_2)$ and write $h_2=\p H$. 

\begin{figure}[t]
\begin{center} 
\includegraphics[height=2in,width=2.8in]{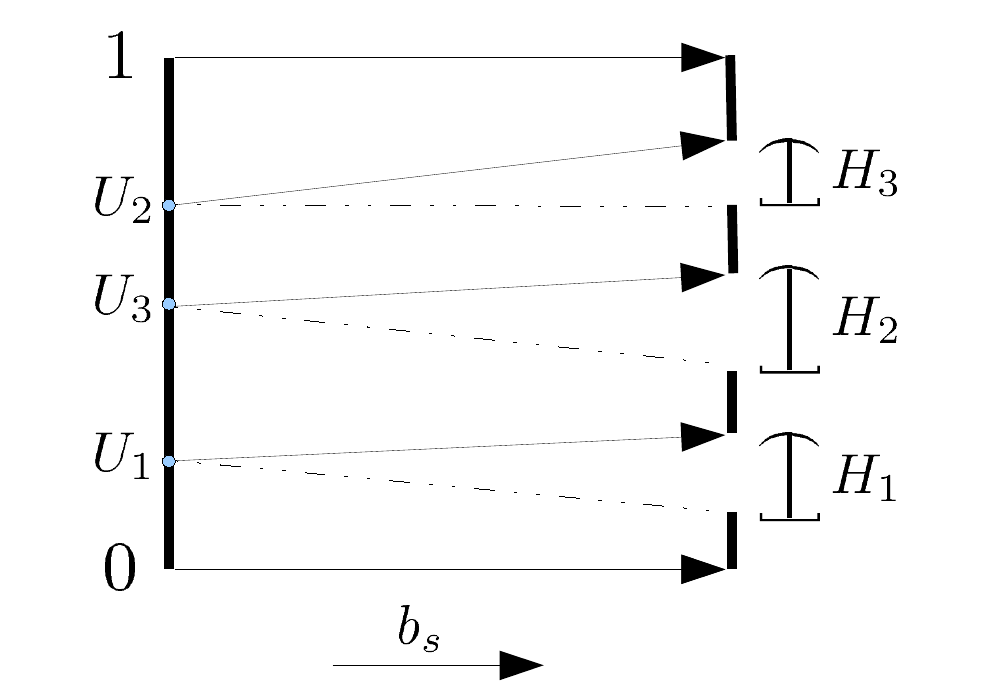} 
\caption{\label{bridge3} \small The holes $H_1,H_2,H_3$ of the bridge $b_s(y)=(1-s_0)y+\sum_1^3s_i\1\{U_i\leq y\}.$}
\end{center} 
\end{figure}

Holes are a natural way of thinking about the connection between bridges and exchangeable partitions, as the following Lemma shows. 

\begin{lemma}\label{holesfreqs} Let $b_s$ be a bridge and let $\pi$ be the exchangeable partition of $\N$ constructed by Lemma \ref{bridgepart} from $b_s$. Then there is a bijective correspondence between the holes of $b_s$ and the atomic blocks of $\pi$, such that
$$\mathscr{S}(H)=\freq(b)$$
(where $b$ is an atomic block of $\pi$ and $H$ is a hole of $b_s$). 
\end{lemma}
\begin{proof}
Recall the construction of Lemma \ref{bridgepart}, which is usually known as Kingman's `paintbox' construction (described in Section 4.4 of \citealt{B2006}). Note that each hole $H$ of size $x$ corresponds to a level set of $b_s^{-1}$, namely 
$$\{y\in[0,1]\-b_s^{-1}(y)=b_s^{-1}(\p H)\}.$$
The probability that each $V_i$ falls into this level set is precisely $x$, which implies that the corresponding block in $\pi$ has asymptotic frequency $x$. 
\end{proof}

The real advantages of thinking in terms of holes is that they provide a way of tracking the atomic blocks through the composition 
$$B^{(n)}_{0,t}=b_{x_{\sigma(1)}}\circ\ldots b_{x_{\sigma(n)}}.$$
We will show precisely how this is achieved in the coming sections. 

\subsection{The case $\nu([0,1])<\infty$}
\label{finitesec}

The case $\nu([0,1])<\infty$ is straightforward and is already well understood, see Example 19 of \cite{P1999}. We will give a heuristic proof using flows of bridges (which is easily made rigorous), as a way of familiarising the reader with some of the methods which we use in Section \ref{dettimesec}.

Since $\nu([0,1])<\infty$, $N_\infty\sim\text{Poisson}(t\nu([0,1]))$, and in particular $\P\l[N_\infty<\infty\r]=1$. Hence, there exists a (random) $M\in\N$ such that, for all $m\geq M$, $$B^{(N_m)}_{0,t}=B^{(N_M)}_{0,t}.$$ By Theorem \ref{lconstrbl}, $\Pi_t$ has the same distribution as the partition associated to $B^{(N_M)}_{0,t}$. Now, since $N_M\leq N_\infty<\infty$, $B^{(N_M)}_{0,t}$ is a finite bridge. Setting $\sigma=\sigma_{N_M}$, we have
$$B^{(N_M)}_{0,t}=b_{x_{\sigma(1)}}\circ\ldots b_{x_{\sigma(N_M)}}.$$
An illustration of the composition of finite bridges is given as Figure \ref{bridge2}.

\begin{figure}[t]
\begin{center} 
\includegraphics[height=3.5in,width=6in]{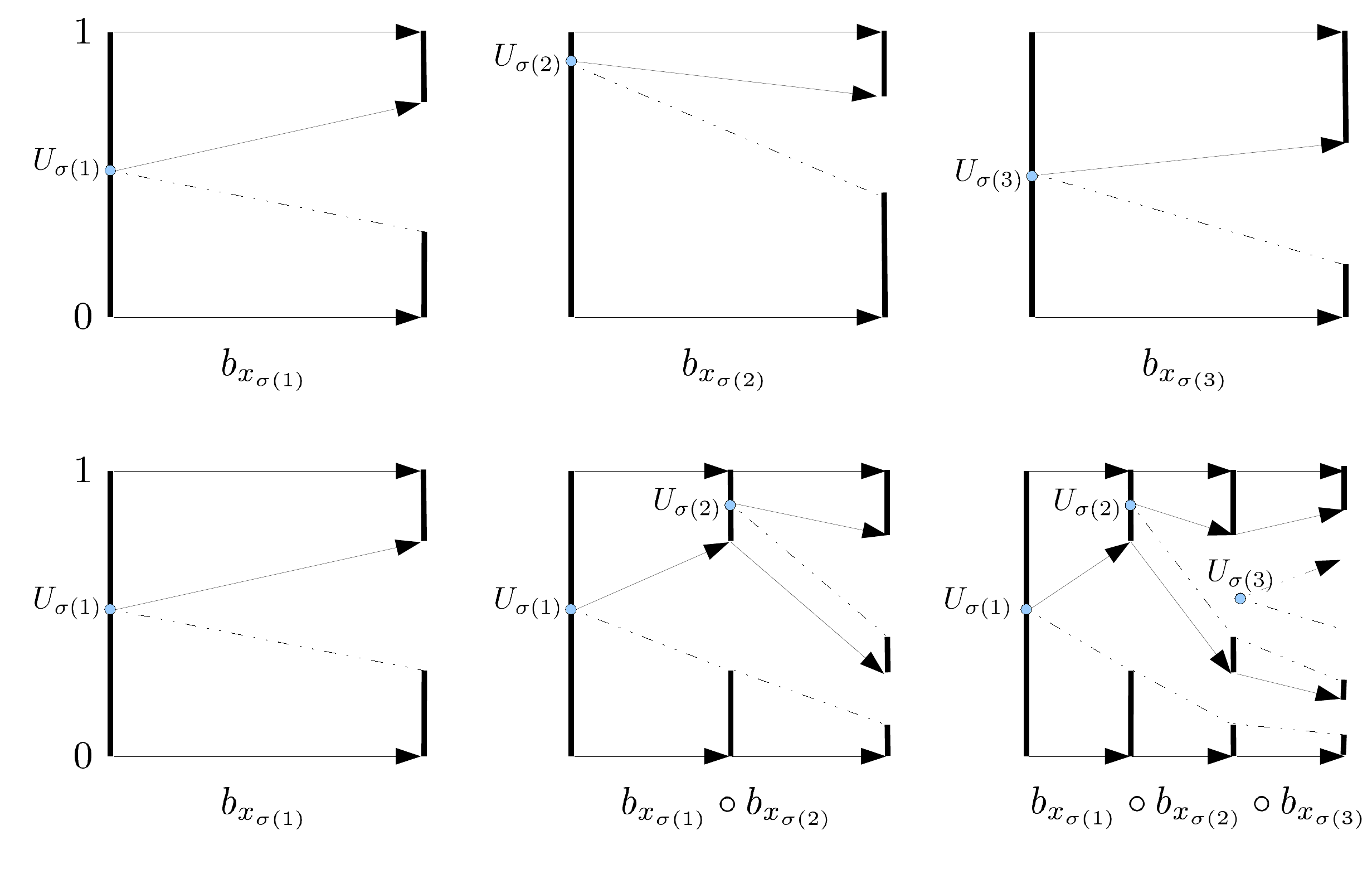} 
\caption{\label{bridge2} \small The composition of three simple bridges. Note that composition with $b_{x_{\sigma(2)}}$ adds a new hole, but composition with $b_{x_{\sigma(3)}}$ does not. Recall that $(f\circ g)(x)=g(f(x))$.}
\end{center} 
\end{figure}

From Figure \ref{bridge2}, we see that composing a new bridge (on the right) with a finite bridge will either (1) keep the number of holes constant, or (2) increase the number of holes by 1. Hence, $B^{(N_M)}_{0,t}$ has at most $N_M<\infty$ holes. By Lemma \ref{holesfreqs}, $\pi^{(N_m)}_t$ has only finite many atomic blocks, and (heuristically at least; there are technicalities to take care of at this point) the limit has the same atomic blocks as $\pi^{(N_m)}_t$, which finishes the argument.

\subsection{At a single deterministic time}
\label{dettimesec}

We now focus on the case $\nu([0,1])=\infty$, which is the new part of the content of Theorem \ref{uscthm}. Recall that $N_\infty=|M|=\infty$, in this case. In this section we prove the required result for a single deterministic time $t\in(0,\infty)$. That is, we prove

\begin{normaltheorem}\label{dettimethm}
Let $t>0$. Then $\P\l[N^a_t=\infty\r]=1$.
\end{normaltheorem}

For the remainder of this section we fix some $t>0$. Before we embark on the proof, we need to collect together some lemmas. For $f:[0,1]\to [0,1]$, let
$$\llip(f)=\sup\{\kappa\in[0,\infty)\-\forall y,z\in[0,1], |f(y)-f(z)|\geq \kappa |y-z|\}$$
be the lower Lipschitz constant of $f$.


\begin{lemma}\label{liplemma} Let $b_s$ be a finite bridge, where $s=(s_i)$. Then $\mathscr{D}(b_s)=\llip(b_{s})$. 
\end{lemma}
\begin{proof}
By definition, we have $b_s(y)=(1-s_0)y+\sum_1^ns_i\1\{U_j\leq y\}$, where $(U_j)$ are uniform random variables on $[0,1]$, independent of both each other and of $s$. Hence $\llip(b_{s})=s_0$. The result follows by Lemma \ref{bridgepart}.
\end{proof}

\begin{lemma}\label{littlelemma}
If $b_s$ is a finite bridge and $b_x$ is a simple bridge then $\mathscr{D}(b_x)=1-x$ and $$\mathscr{D}(b_s\circ b_s)=\mathscr{D}(b_s)\mathscr{D}(b_x).$$ 
\end{lemma}
\begin{proof}
Again, by definition, we have $b_s(y)=(1-s_0)y+\sum_1^ns_i\1\{U_j\leq y\}$, where $(U_j)$ are uniform random variables on $[0,1]$, independent of both each other and of $s$. If we write the $U_i$ in increasing order as $U_{\tau(i)}$, then for each $y,z\in\big[U_{\tau(i)},U_{\tau(i+1)}\big)$, 
\begin{equation}\label{contract}
|b_s(y)-b_s(z)|=s_0|y-z|.
\end{equation} The result follows immediately from \eqref{contract} and Lemma \ref{liplemma}, since composing (local) isometries corresponds to multiplying $\llip$.
\end{proof}

We now work towards Theorem \ref{dettimethm}.

\begin{lemma}\label{dustinfo}
The following hold. 
\begin{enumerate}
\item For all $n\in\N$, $\mathscr{D}(B^{(n)}_{0,t})\geq \mathscr{D}(B^{(n+1)}_{0,t})$. 
\item The limit $D_t=\lim_{m\to\infty}\mathscr{D}(B^{(N_m)}_{0,t})$ exists almost surely.
\item $D_t$ and $\mathscr{D}(\Pi_t)$ have the same distribution.
\end{enumerate}
\end{lemma}
\begin{proof}
Using Lemma \ref{littlelemma} iteratively on \eqref{Bnt},
\begin{equation}\label{dusteq}
\mathscr{D}\l(B^{(n)}_{0,t}\r)=\prod\limits_{i=1}^n(1-x_{\sigma(i)})
\end{equation}
where $\sigma=\sigma_n$. Statement \textit{1} follows immediately. Hence, the limit $D_t=\lim_{m\to\infty}\mathscr{D}(B^{(N_m)}_{0,t})$ exists almost surely (since monotone decreasing sequences converge), which proves \textit{2}. Note that in fact $D_t=\lim_{n\to\infty}\mathscr{D}(B^{(n)}_{0,t})$ (although we will not have need of this fact).

Note that \textit{3} concerns only the distribution $\Pi_t$ and $\pi^{(\infty)}_t$. Let us write $$B^{(N_m)}_{0,t}=b_{(s^m_i)_{i=1}^\infty},$$ so that $|\pi^{(N_m)}_t|^{\downarrow}=(s^m_i)$, and similarly write $|\Pi_t|^{\downarrow}=(s^\infty_i)$. By Theorem \ref{lconstrbl}, $|\pi^{(N_m)}_t|^{\downarrow}\to |\Pi_t|^{\downarrow}$ in distribution. By the compactness of $\mc{P}_\mc{M}$, the Skorohod Representation Theorem implies that we may assume (in as far as proving \textit{3} is concerned) that $|\pi^{(N_m)}_t|^{\downarrow}\to |\Pi_t|^{\downarrow}$ almost surely.  Hence, by Lemma \ref{skoruse}, $s^m_i\to s^\infty_i$, for all $i\in\N$. By Lemma \ref{bridgepart}, 
$$\mathscr{D}(\Pi_t)=1-\sum_1^\infty s^\infty_i,$$
and by Lemma \ref{liplemma},
$$\mathscr{D}(\pi^{(N_m)}_t)=1-\sum_1^\infty s^m_i.$$
By dominated convergence, almost surely, $\lim_m 1-\sum_1^\infty s^m_i=1-\sum_1^\infty s^\infty_i,$
which completes the proof.
\end{proof}

\begin{remark}
A change of probability space occurs in the above proof, caused by the application of Skorohod's Representation Theorem. This change preserves the distribution of each $\pi^{(N_m)}_t$, and the limit $\Pi_t$, individually, but does not preserve their joint distribution. We will use the joint distribution in the sequel, so, outside of the above proof, we do not change our background probability space. Consequently, in \textit{3} of Lemma \ref{dustinfo}, we record only the result that $\mathscr{D}(\Pi_t)$ is equal to $D_t$ in distribution (and not almost surely).
\end{remark}

\begin{lemma}\label{Ulemma}
Let $b_s$ be a finite bridge, and let $V$ be a uniform random variable on $[0,1]$ which is independent of $b_s$. Let $(H_i)$ be the (finite sequence of) holes of $b_s$, and let $H=\cup_i H_i$. Then there exists a uniform random variable $U$ such that $U< \mathscr{D}(b_s)$ if and only if $V\notin H$.
\end{lemma}
\begin{proof}
Note that if $\mathscr{D}(b_s)=0$ then the result is trivial, so assume $\mathscr{D}(b_s)>0$. Let $H_1,\ldots, H_l$ be the holes of $b_s$, numbered such that $\p H_i< \p H_j$ if $i< j$. By definition, the holes are disjoint. For each hole, let $\delta H_j=\inf H_j$, recall $\p H_j=\sup H_j$, and set $\p H_0=0$, $\delta H_{l+1}=1$.

For $j=0,1,\ldots, l$, set $I_j=[\p H_j, \delta H_{j+1})$. Define a function $f:[0,1]\to[0,1]$ as follows.
\begin{equation*}
f(y)=\begin{cases}
\sum_{i=0}^{j-1}(\delta H_{i+1}-\p H_i)+(y-\delta H_{j+1}) & \text{ if }y\in I_j\\
\sum_{i=0}^{l}(\delta H_{i+1}-\p H_i)+(y-\delta H_{j+1})+\sum_{i=1}^{j-1}(\p H_i-\delta H_i) +(y-\p H_j) & \text{ if }y\in H_j\\
1 & \text{ if }y=1.
\end{cases}
\end{equation*}
A graphical demonstration of $f$ is given in Figure \ref{bridge4}. It is easily checked that $U=f(V)$ has the required properties.
\end{proof}

\begin{figure}[t]
\begin{center} 
\includegraphics[height=3.5in,width=5in]{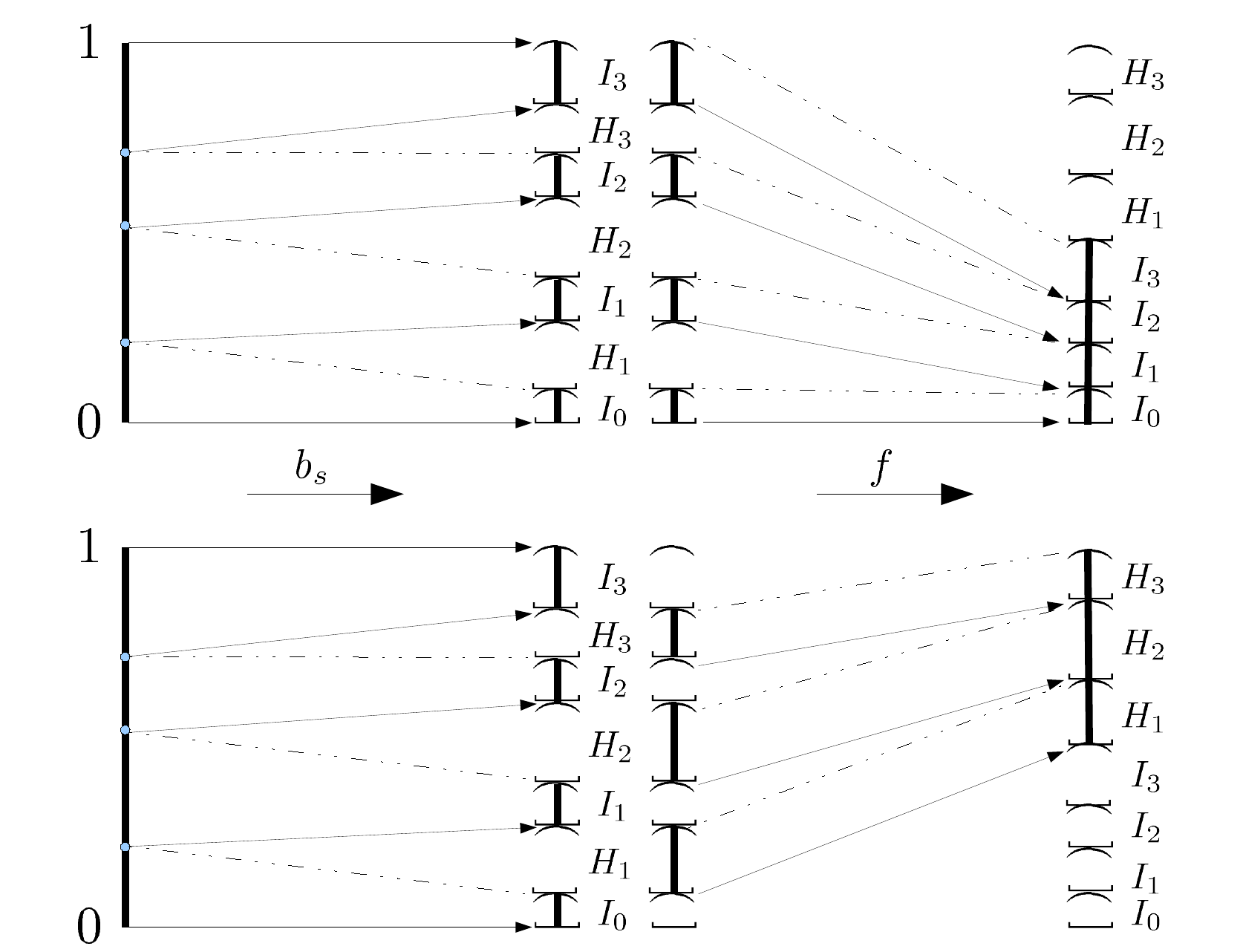} 
\caption{\label{bridge4} \small Diagram of $f$, from the proof of Lemma \ref{Ulemma}. The action of $f$ is first to stack semgents of dust on top of each other, and then stack the holes on top of that.}
\end{center} 
\end{figure}

\begin{lemma}\label{lotsofholes} Let $m\in\N$. Then there exists a sequence $(U_i)_1^\infty$ of independent uniform random variables on $[0,1]$ such that the following holds. Almost surely, for all $j=1,\ldots, m$, the bridge $B^{(N_m)}_{0,t}$ has at least 
$$\sum_{i=1}^{N_{j}}\1\{U_i\leq D_t\}$$
holes of size at least $\dfrac{D_t}{j}$.
\end{lemma}
\begin{proof}
Note that a finite bridge has only finite many holes. Fix $j\in\{1,\ldots, n\}$, and suppose $H_1,\ldots, H_l$ are the holes of some finite bridge $b_s=b_{s_1,\ldots,s_l}$. Let $x\in (0,1)$, and let $V$ denote a uniform random variable (independent of $b_s$ and $x$) such that
$$b_{x}(a)=(1-x)a+x\1\{V\leq a\}.$$
We consider what happens to the holes of $b_s$ when we compose with $b_x$ to form $b_s\circ b_x$.

The probability of $V$ falling into the set $\{\p H_i\-i=1,\ldots,l\}$ is zero, and henceforth we ignore this case (we claim only an almost sure result in the statement of this lemma). There are two further options:
\begin{enumerate}
\item[(A)] If $V\in [0,1]\sc \l(\bigcup_i H_i\cup\{\p H_i\}\r)$, then the holes of $b_s\circ b_x$ are precisely the elements of the set $$\{b_x(H_i)\-i=1,\ldots,l\}\cup\Big\{\Big[(1-x)V,\,(1-x)V+x\Big)\Big\}.$$ 
\item[(B)] If $V\in H_k=[h^k_1,h^k_2)$, for some $k$, then the holes of $b_s\circ b_x$ are precisely the elements of the set $$\big\{b_x(H_i)\-i=1,\ldots,k-1,k+1,\ldots,l\big\}\cup\Big\{\Big[b_x(h^k_1),\,b_x(h^k_2)\Big)\Big\}.$$
\end{enumerate}

\begin{remark}
The reader might wish to refer back to Figure \ref{bridge2}, where the two cases can both be seen. In that figure, (A) occurs when composing with $b_{x_{\sigma(1)}}$ with $b_{x_{\sigma(2)}}$, and (B) occurs when composing $b_{x_{\sigma(1)}}\circ b_{x_{\sigma(2)}}$ with $b_{x_{\sigma(3)}}$.   
\end{remark}

In case (A), for all $i=1,\ldots, l$ we say that the hole $b_x(H_i)$ of $b_s\circ b_x$ is a \emph{child} of $H_i$. In case (B), for $i=1,\ldots, k-1,k+1$, we say also that the hole $b_x(H_i)$ of $b_s\circ b_x$ is a child of $H_i$. Additionally, in case (B) we say that $[b_x(h^k_1),\,b_x(h^k_2))$ is a child of $H_k$. 

We label the following results for easy reference.
\begin{enumerate}
\item[(1)] Every hole of $b_s$ has a unique child amongst holes of $b_s\circ b_x$.
\item[(2)] If a hole $H'$ of $b_s\circ b_x$ is a child of the hole $H$ of $b_s$, then $\mathscr{S}(H')\geq (1-x)\mathscr{S}(H)$.
\item[(3)] In case (A), $b_s\circ b_x$ has one more hole than $b_s$, and in case (B) $b_s\circ b_x$ has the same number of holes as $b_x$. In case (A), the new hole has size $x$.
\end{enumerate}
The first of the above statements is essentially immediate, since $b_x$ is a strictly increasing function with (by Lemma \ref{liplemma}) $\llip(b_x)>0$. The third statement is immediate from (A) and (B). The second statement follows from Lemma \ref{liplemma}, except in case (B) for the hole $H=H_k=[h^k_1,h^k_2)$. In this case $H'=\big[b_x(h^k_1),b_x(h^k_2)\big)$, and a direct calculation shows that
$$\mathscr{S}(H')=b_x(h_2^k)-b_x(h_1^k)=(1-x)(h^k_2-h^k_1)+x=(1-x)\mathscr{S}(H)+x.$$

By Lemma \ref{holesfreqs}, the sum of the sizes of the holes of $b_s$ is precisely $1-\mathscr{D}(b_s)=s_0$ where $s_0=1-\sum_1^\infty s_i$. By Lemma \ref{Ulemma}, there is a random variable $U\in\sigma(V)$ such that 
\begin{equation}\label{Udef}
\1\{U< \mathscr{D}(b_s)\}=\1\{V\text{ is not an element of a hole of }b_s\}.
\end{equation}
Since $V$ is independent of $\mathscr{D}(b_s)$, $U$ is independent of $\mathscr{D}(b_s)$. We record a fourth statement, which is also immediate.
\begin{itemize}
\item[(4)] Case (A) occurs precisely when \eqref{Udef}$=1$, and case (B) occurs precisely when \eqref{Udef}$=0$.
\end{itemize}

Now, fix $m\in\N$, and let $\sigma=\sigma_{N_m}$. Loosely speaking, we will apply the above reasoning iteratively along the composition $b^{(N_m)}=b_{x_{\sigma(1)}}\circ\ldots\circ b_{x_{\sigma(N_m)}}$, working from left to right. For each $i=2,\ldots,N_m$, at the $i^{th}$ stage we set
\begin{align*}
b_s&=b_{x_{\sigma(1)}}\circ\ldots\circ b_{x_{\sigma(i-1)}}\\
b_x&=b_{x_{\sigma(i)}}.
\end{align*}
We divide the remainder of the argument into two stages.

\textsc{Construction of $(U_i)$:} Let $(V_i)_1^\infty$ be a sequence of independent uniform random variables on $[0,1]$ which are independent of $\sigma(\delta(M\cap(0,t)\times(0,1)))$, and use $V_i$ as the uniform random variable (written as $V$ above), at stage $i$ of the iteration. It follows immediately that $U_i$, defined (at each step of the iteration) by \eqref{Udef}, is also a sequence of independent uniform random variables on $[0,1]$.

\textsc{Construction of the holes:} Fix $j\in\{1,\ldots,m\}$. Recall that the sets $\{x_i\-i=1,\ldots, N_j\}$ and $\{x_{\sigma(i)}\-i=1,\ldots, N_j\}$ are in bijective correspondence. Hence, by definition of $N_j$, 
\begin{itemize}
\item[(5)] There are at least $N_j$ elements of $\{x_{\sigma(i)}\-i=1,\ldots, N_j\}$ for which $x_{\sigma(i)}>1/j$. 
\end{itemize}
Let $i'$ be some $i'\in\{1,\ldots, N_j\}$ such that $x_{\sigma(i')}>1/j$. If $V_{i'}$ falls into a hole of $$b_s=b_{x_{\sigma(1)}}\circ\ldots\circ b_{x_{\sigma(i'-1)}}$$ (or equivalently, by (4), if $U_i\leq \mathscr{D}(b_s)$), then, by (3), the new hole $H$ of $b_s\circ b_{x_{\sigma(i')}}$ has size $x_{\sigma(i')}\geq 1/j$. By (1), following the line of children of $H$, we reach a hole $\wt{H}$ of $b^{(N_m)}_{0,t}$, and by (2) 
$$\mathscr{S}\l(\wt{H}\r)\geq x_{\sigma(i')}\prod\limits_{i=i'+1}^{N_m}(1-x_{\sigma(i)}).$$
By \eqref{dusteq}, 
$$\prod_{i=i'+1}^{N_m}(1-x_{\sigma(i)})\geq \prod_{i=1}^{N_m}(1-x_{\sigma(i)})=\mathscr{D}\l(B^{(N_m)}_{0,t}\r).$$
By Lemma \ref{dustinfo}, $\mathscr{D}(B^{(N_m)}_{0,t})\geq D_t$, so in fact
\begin{equation}\label{holesize}
\mathscr{S}\l(\wt{H}\r)\geq x_{\sigma(i')}D_t\geq \frac{D_t}{j}.
\end{equation}

To summarise, by (5), there are at least $N_j$ possible distinct choices for $i'\in\{1,\ldots,N_m\}$ for which $x_{\sigma(i')}>1/j$. By (4), when $U_{i'}\leq \mathscr{D}(b_s)$, each such $i'$ gives rise to a hole of $B^{(N_m)}_{0,t}$ which satifies \eqref{holesize}. By the uniqueness of children in (1), each $i'$ for which $U_{i'}\leq \mathscr{D}(b_s)$ gives rise to a unique hole of $B^{(N_m)}_{0,t}$ (which, of course, satisfies \eqref{holesize}). Since the $(U_i)$ are independent of $N_j$, this completes the proof.
\end{proof}

\begin{lemma}\label{closedsets}
For all $\epsilon>0$ and $N\in\N$, the set
$$\mathcal{B}(N,\epsilon)=\{(s_i)_1^\infty\in\mc{P}_{\mc{M}}\-\forall i=1,\ldots,N,\; s_i\geq \epsilon\}$$
is a closed subset of $\mc{P}_{\mc{M}}$.
\end{lemma}
\begin{proof}
This follows from Lemma \ref{skoruse}, since $\mathcal{B}(N,\epsilon)$ is closed under pointwise convergence.
\end{proof}

\begin{proof}[Of Theorem \ref{dettimethm}.]
The argument rests on an application of the Portmanteau Theorem, which can be found as Theorem 3.1 in \cite{EK1986}. The precise fact we require is the following.
\begin{itemize}
\item[($\star$)] Let $(S,d_S)$ be a separable metric space, and let $\Q_n,\Q$ (where $n\in\N$) be probability measures on the Borel subsets of $S$. Let $X_n$ and $X$ be $S$-valued random variables with distributions $\Q_n$ and $\Q$, repectively. Then, $X_n\to X$ in distribution if and only if, for all closed subsets $C$ of $S$, $\limsup_n \Q_n(C)\leq \Q(C)$.
\end{itemize}
Of course, we will use the closed sets described by Lemma \ref{closedsets}, and the convergence in distribution from Theorem \ref{lconstrbl}.

Let $\delta>0$ and let $N\in\N$. By Lemma \ref{dustinfo}, $D_t$ has the same distribution as $\mathscr{D}(\Pi_t)$, and by Theorem \ref{sthm}, $\P\l[\mathscr{D}(\Pi_t)>0\r]=1$. 
Hence, we may choose $\epsilon>0$ such that 
\begin{equation}\label{eqqq1}
\P\l[D_t\geq \epsilon\r]\geq 1-\delta.
\end{equation} Since $\nu([0,1])=\infty$ we have that $\lim_{j\to\infty}N_j=\infty$, almost surely. Hence, we may pick $j\in\N$ such that 
\begin{equation}\label{eqqq2}\P\l[N_j\geq N^2\r]\geq 1-\delta.
\end{equation} 
We will use \eqref{eqqq1} and \eqref{eqqq2} frequently in the following estimates. For all $m\geq j$,
\begin{align*}
&\P\l[B^{(N_m)}_{0,t}\text{ has at least }N\text{ holes of size}\geq\frac{\epsilon}{j}\r]\\
&\hspace{4pc}\geq \P\l[B^{(N_m)}_{0,t}\text{ has at least }N\text{ holes of size}\geq\frac{D_t}{j}\r]-\delta\\
&\hspace{4pc}\geq \P\l[B^{(N_m)}_{0,t}\text{ has at least }\sum_1^{N_j}\1\{U_i\leq D_t\}\text{ holes of size}\geq\frac{D_t}{j}\r]\\
&\hspace{8pc}-\P\l[\sum_1^{N_j}\1\big\{U_i\leq D_t\big\}\leq N\r]-\delta\\
&\hspace{4pc}=1-\P\l[\sum_1^{N_j}\1\big\{U_i\leq D_t\big\}\leq N\r]-\delta.
\end{align*}
To get from the penultimate to the final line of the above we use Lemma \ref{lotsofholes}. Now, 
\begin{align*}
\P\l[\sum_1^{N_j}\1\big\{U_i\leq D_t\big\}\leq N\r]&\leq \P\l[\sum_1^{N^2}\1\big\{U_i\leq D_t\big\}\leq N\r]+\delta\\
&\leq \P\l[\sum_1^{N^2}\1\big\{U_i\leq\epsilon\big\}\leq N\r]+2\delta
\end{align*}
Note that $\sum_1^{N^2}\1\big\{U_i\leq\epsilon\big\}$ is just a binomial random variable with $N^2$ trials and success probability $\epsilon$. Applying Hoeffding's inequality, we obtain
\begin{align*}
\P\l[\sum_1^{N^2}\1\big\{U_i\leq\epsilon\big\}\leq N\r]\leq \frac{1}{2}\exp\l(-2\frac{(N^2\epsilon-N)^2)}{N^2}\r)=\frac{1}{2}\exp\l(-2(N\epsilon-1)^2\r)
\end{align*}
Collecting together, we have that for all $m\geq j$,
\begin{align}\label{iieq3}
\P\l[B^{(N_m)}_{0,t}\text{ has at least }N\text{ holes of size}\geq\frac{\epsilon}{j}\r]\geq 1-\frac{1}{2}\exp\l(-2(N\epsilon-1)^2\r)-3\delta
\end{align}
Equation \eqref{iieq3} is the crucial bound, and we can now move towards applying ($\star$). Note that the right hand side of \eqref{iieq3} does not depend on $m$. 

Let $\Q_m$ and $\Q$ be the laws of $\pi^{(N_m)}_t$ and $\Pi_t$, respectively. Then, by Lemma \ref{holesfreqs}
$$\P\l[B^{(N_m)}_{0,t}\text{ has at least }N\text{ holes of size}\geq\frac{\epsilon}{j}\r]=\Q_m(\mc{B}(N,\epsilon/j)).$$
Similarly, 
$$\Q(\mc{B}(N,\epsilon/j))=\P\Big[\Pi_t\text{ has at least }N\text{ blocks with asymptotic frequency}\geq\epsilon/j\Big].$$
By Theorem \ref{lconstrbl}, $\pi^{(N_m)}_t\to\Pi_t$ in distribution, and by ($\star$), and Lemma \ref{closedsets},
$$\limsup\limits_{m\to\infty}\Q_m(\mc{B}(N,\epsilon/j))\leq \Q(\mc{B}(N,\epsilon/j)).$$
Hence, by \eqref{iieq3},
\begin{samepage}
\begin{align*}
&\P\Big[\Pi_t\text{ has at least }N\text{ blocks with asymptotic frequency}\geq\epsilon/j\Big]\\
&\hspace{19pc}\geq 1-\frac{1}{2}\exp\l(-2(N\epsilon-1)^2\r)-3\delta.
\end{align*}
\end{samepage}
Recall that $j$ depends on $N$, but $\delta$ and $\epsilon$ do not depend on $N$. Recall also that $\epsilon$ depends on $\delta$, but $j$ and $N$ do not depend on $\delta$. From the above equation we have
$$\P\Big[\Pi_t\text{ has at least }N\text{ blocks with asymptotic frequency}>0\Big]\geq 1-\frac{1}{2}\exp\l(-2(N\epsilon-1)^2\r)-3\delta,$$
and letting $N\to\infty$,
$$\P\Big[\Pi_t\text{ has infinitely many blocks with asymptotic frequency}>0\Big]\geq 1-3\delta.$$
Since $\delta>0$ was arbitrary, we thus have
$$\P\Big[\Pi_t\text{ has infinitely blocks with asymptotic frequency}>0\Big]=1,$$
and the proof is complete.
\end{proof}

\subsection{An embedded $\Lambda$-coalescent}

In this section we show that, in a sense made precise by Theorem \ref{lsub}, the restriction of the $\Lambda$-coalescent to any infinite subset of its particles is also a $\Lambda$-coalescent. We then combine this with Theorem \ref{dettimethm} to complete the proof of Theorem \ref{uscthm}.

In this section we will consider $\Lambda$-coalescents with initial time $t>0$ (instead of $t=0$, which was specified by Definition \ref{lambdacoaldef}). We extend Definition \ref{lambdacoaldef} in the obvious manner.

Let $\mc{F}_t=\sigma(\Pi_s\-s\leq t)$. 

\begin{normaltheorem}\label{lsub}
Let $T>0$, and let $\mc{W}=\{w_i\-i\in\N\}$ be a random, $\mc{F}_T$ measurable, subset of $\N$, such that $w_i< w_j$ if $i<j$. Suppose that $\mc{W}$ contains precisely one element of each block $b\in\Pi_T$.

For any $n\in\N$, let $\bl_{\Pi_t}(n)$ be the (unique) block $b\in\Pi_t$ such that $n\in b$. Define a process $(\wt{\Pi}_t)_{t\geq T}$ taking values in $\mc{P}_\N$ by
\begin{equation}\label{iieq4}
i\stackrel{\wt{\Pi}_t}{\sim} j\Longleftrightarrow \bl_{\Pi_t}(w_i)=\bl_{\Pi_t}(w_j).
\end{equation}
Then $(\wt{\Pi}_t)_{t\geq T}$ is a $\Lambda$-coalescent.
\end{normaltheorem}

In order to prove Theorem \ref{lsub} we will need the following result.

\begin{normaltheorem}[\citealt{P1999}]\label{strongmarkov}
The $\Lambda$-coalescent is strongly $\mc{F}_t$-Markov.
\end{normaltheorem}

\begin{proof}[Of Theorem \ref{lsub}]
Note that the map $i\mapsto w_i$ is a bijection from $\N\to \N$, and hence $\wt{\Pi}_t$ is well defined. Define $\eta:\N\to\N$ by $\eta(w_i)=i$. The initial state of $\wt{\Pi}_t$ is $\wt{\Pi}_T=\1_\N$, since $\mc{W}$ contains precisely one element of each block of $\Pi_t$. 

Let $\wt{\Pi}_i^n=\iota_n(\wt{\Pi}_t)$. Suppose that $\wt{\Pi}_i^n$ has $l\in\N$ blocks, and consider any subset $\wt{b}_1,\ldots, \wt{b}_k$ of distinct blocks of $\wt{\Pi}_t$, where $t\geq T$. For each $\wt{b}_i$, (by \eqref{iieq4},
$$b_i=\bigcup\limits_{n\in \wt{b}_i}\bl_{\Pi_t}(n)$$
is a block of $\Pi_t$, and for $i\neq j$, $b_i\cap b_j=\emptyset$. Hence, by Theorem \ref{strongmarkov} (applied at time $t$), and Definition \ref{lambdacoaldef}, the rate at which the $k$-tuple of of blocks $\wt{b}_1,\ldots,\wt{b}_k$ is 
$$\lambda_{l,k}=\int_0^1x^{k}(1-x)^{l-k}\nu(dx).$$ It is straightforward to see that this same rate of coagulation occurs independently for all $k$-tuples of $\wt{b}_1,\ldots,\wt{b}_l$. By Definition \ref{lambdacoaldef}, $(\wt{\Pi}_t)_{t\geq T}$ is a $\nu$-coalecsent.
\end{proof}

We are now in a position to prove our main result, Theorem \ref{uscthm}.

\begin{proof}[Of Theorem \ref{uscthm}] Let $T>0$ be deterministic and set
$$\mc{W}=\{\inf b\- b\in\Pi_T, |b|>1\}.$$
Note that, by Theorem \ref{dettimethm}, $\mc{W}$ satisfies the conditions for Theorem \ref{lsub}. By Theorem \ref{lsub}, the corresponding coalescent $(\wt{\Pi}_t)_{t\geq T}$, defined by \eqref{iieq4}, is a $\nu$-coalescent. By Theorem \ref{sthm} (which applies since $\mu^{-1}<\infty$ implies $\mu^*=\infty$), $(\wt{\Pi}_t)_{t\geq T}$ does not come down from infinity. That is, $$\P\l[\forall t\geq T, |\wt{\Pi}_t|=\infty\r]=1.$$ By \eqref{iieq4}, each block of $\wt{\Pi}_t$ corresponds uniquely to an atomic block of $\Pi_t$, which implies that $\P\l[\forall t\geq T, N^a_t=\infty\r]=1$. Since this holds for each $T=1/n$, where $n\in\N$, we have that $$\P\l[\forall t>0, N^a_t=\infty\r]=1.$$
This completes the proof.
\end{proof}

\subsection*{Acknowledgement}

I am very grateful to both Alison Etheridge and Vlada Limic; the idea for this paper came out of conversation between the three of us. To the best of my knowledge, Theorem \ref{uscthm} was first conjectured by Vlada Limic, who also gave the first example/proof of a $\Lambda$-coalescent with $N^a_t=N^s_t=\infty$ (in an unpublished note, using different methods to the proofs above).

\bibliographystyle{plainnat}
\bibliography{biblio1}

\begin{thebibliography}{14}
\providecommand{\natexlab}[1]{#1}
\providecommand{\url}[1]{\texttt{#1}}
\expandafter\ifx\csname urlstyle\endcsname\relax
  \providecommand{\doi}[1]{doi: #1}\else
  \providecommand{\doi}{doi: \begingroup \urlstyle{rm}\Url}\fi

\bibitem[Berestycki et~al.(2008)Berestycki, Berestycki, and
  Schweinsberg]{BBS2008}
J.~Berestycki, N.~Berestycki, and J.~Schweinsberg.
\newblock Small-time behavior of {$\beta$}-coalescents.
\newblock \emph{Ann. Inst. H. Poincare Probab. Statis.}, 44\penalty0
  (2):\penalty0 214--238, 2008.

\bibitem[Berestycki(2009)]{B2009}
N.~Berestycki.
\newblock \emph{Recent Progress In Coalescent Theory}, volume~16.
\newblock Ensaios Matematicos, 2009.

\bibitem[Bertoin(2006)]{B2006}
J.~Bertoin.
\newblock \emph{Random Fragmentation and Coagulation Processes}.
\newblock Cambridge University Press, 2006.

\bibitem[Bertoin and Le~Gall(2003)]{BL2003}
J.~Bertoin and J.~F. Le~Gall.
\newblock Stochastic flows associated to coalescent processes.
\newblock \emph{Probab. Th. Rel. Fields}, 126:\penalty0 261--288, 2003.

\bibitem[Birkner et~al.(2005)Birkner, Blath, Capaldo, Etheridge, M\"{o}hle,
  Schweinsberg, and Wakolbinger]{BBC2005}
M.~Birkner, J.~Blath, M.~Capaldo, A.~M. Etheridge, M.~M\"{o}hle,
  J.~Schweinsberg, and A.~Wakolbinger.
\newblock {$\alpha$}-stable branching and {$\beta$}-coalescents.
\newblock \emph{Electron. J. Probab.}, 10\penalty0 (9):\penalty0 303--325,
  2005.

\bibitem[Bolthausen and Sznitman(1998)]{BS1998}
E.~Bolthausen and A.~Sznitman.
\newblock On {R}uelle's probability cascades and an abstract cavity method.
\newblock \emph{Comm. Math. Phys.}, 197:\penalty0 247--276, 1998.

\bibitem[Donnelly and Kurtz(1999)]{DK1999}
P.~Donnelly and T.~G. Kurtz.
\newblock Genealogical processes for {F}leming-{V}iot models with selection and
  recombination.
\newblock \emph{Ann. Appl. Probab.}, 9\penalty0 (4):\penalty0 1091--1148, 1999.

\bibitem[Ethier and Kurtz(1986)]{EK1986}
S.~Ethier and T.~G. Kurtz.
\newblock \emph{Markov Processes: Characterization and Convergence}.
\newblock Wiley, 1986.

\bibitem[Freeman(2011)]{F2011a}
N.~Freeman.
\newblock Phase transitions in a spatial coalescent.
\newblock \emph{http://arxiv.org/abs/1109.4363}, pages 1--44, 2011.

\bibitem[Kingman(1982)]{K1982}
J.~F.~C. Kingman.
\newblock The coalescent.
\newblock \emph{Stochastic Process. Appl.}, 13:\penalty0 235--248, 1982.

\bibitem[Pitman(1999)]{P1999}
J.~Pitman.
\newblock Coalescents with multiple collisions.
\newblock \emph{Ann. Probab.}, 27\penalty0 (4):\penalty0 1870--1902, 1999.

\bibitem[Sagitov(1999)]{S1999}
S.~Sagitov.
\newblock The general coalescent with asynchronous mergers of ancestral lines.
\newblock \emph{Journal of Applied Probability}, 36:\penalty0 1116--1125, 1999.

\bibitem[Schweinsberg(2000)]{S2000}
J.~Schweinsberg.
\newblock A necessary and sufficient condition for the {$\Lambda$}-coalescent
  to come down from infinity.
\newblock \emph{Electron. Comm. Probab.}, 5:\penalty0 1--11, 2000.

\bibitem[Schweinsberg(2003)]{S2003}
J.~Schweinsberg.
\newblock Coalescent processes obtained from supercritical {G}alton-{W}atson
  processes.
\newblock \emph{Stochastic Process. Appl.}, 106:\penalty0 107--139, 2003.

\end{thebibliography}

\end{document}